\documentclass[12pt]{amsart}

\usepackage[margin=3.5 cm]{geometry}
\newtheorem{theorem}{Theorem}[section]
\newtheorem{prop}{Proposition}[section]

\newtheorem{lemma}[theorem]{Lemma}

\theoremstyle{definition}

\theoremstyle{remark}
\newtheorem{remark}[theorem]{Remark}

\numberwithin{equation}{section}

\newcommand{\abs}[1]{\lvert#1\rvert}
\newcommand{\norm}[1]{\lVert#1\rVert}

\DeclareMathOperator{\im}{Im}
\newcommand{\ud}{\mathrm{d}}

\begin{document}

\title[Singular spectral measures for ergodic Jacobi matrices]{Singular components of spectral measures for ergodic Jacobi matrices}

\author{C. A. Marx}
\address{Department of Mathematics, University of California, Irvine CA, 92717}

\thanks{The work was supported by NSF Grant DMS - 0601081 and BSF, grant 2006483.}



\begin{abstract}
For ergodic 1d Jacobi operators we prove that the random singular components of any spectral measure are a.s. mutually disjoint as long as one restricts to the set of positive Lyapunov exponent. In the context of extended Harper's equation this yields the first rigorous proof of the Thouless' formula for the Lyapunov exponent in the dual regions.
\end{abstract}

\maketitle

\section{Introduction} \label{sec_intro}

Given two sequences 
\begin{eqnarray}
(a_n) \in \mathit{l}^\infty(\mathbb{Z};\mathbb{C}) ~\mbox{,} \abs{a_n} > 0 ~\forall n \in \mathbb{Z} ~\mbox{,} \\
(b_n) \in \mathit{l}^\infty(\mathbb{Z};\mathbb{R}) ~\mbox{,}
\end{eqnarray}
we define 
\begin{equation} \label{eq_Jacobiop_1}
(H \psi)_n := b_{n} \psi_{n} + a_{n} \psi_{n+1} + \overline{a_{n-1}} \psi_{n-1} ~\mbox{,} 
\end{equation}
a bounded self adjoint operator on $\mathcal{H}:=\mathit{l}^2(\mathbb{Z};\mathbb{C})$.

Operators of the form (\ref{eq_Jacobiop_1}) are known as 1 d - Jacobi operators (or Jacobi matrices), the most prominent and well studied case being Schr\"odinger operators where  $(a_n) = (1)$. 

We note that since (\ref{eq_Jacobiop_1}) is unitarily equivalent to a Jacobi operator with $a_n>0, ~\forall n$, one could assume all $a_n >0$ right from the start. In some situations it may however be more convenient to retain the original complex form, which is why we do not apply this unitary here. 

In this article, we shall be interested in an ergodic family of Jacobi operators, $(H_\omega)_{\omega \in \Omega}$, induced by random variables
$a,b$ on a probability space $(\Omega, \mathcal{F}, \mathbb{P})$,
\begin{eqnarray} 
a: \Omega \to \mathbb{C}, ~b:  \Omega \to \mathbb{R} ~\mbox{,}\\
\left\vert \mathbb{E}_\omega \log \abs{a} \right\vert < \infty ~\mbox{,} \label{eq_Jacobiop_2}
\end{eqnarray}
and an ergodic invertible map $T$ on $\Omega$ generating the random sequences $(a_n(\omega))$ and $(b_n(\omega))$ according to
\begin{equation} \label{eq_stationarity}
a_n(\omega) := a(T^n \omega) ~\mbox{,} ~ b_n(\omega) = b(T^n \omega) ~\mbox{.}
\end{equation}
\begin{remark} \label{rem_ergodic}
In particular,  (\ref{eq_stationarity}) together with (\ref{eq_Jacobiop_2})  imply that $a_n(\omega) \neq 0, ~\forall n \in \mathbb{Z}$, $\mathbb{P}$-a.s.
\end{remark}

We emphasize that in (\ref{eq_Jacobiop_1}) we do not require a {\em{positive}} lower bound for the sequence $(a_n)$. From a dynamical point of view, fixing $z \in \mathbb{C}$, this manifests itself in almost singularity of the 1-step transfer matrices
\begin{equation} \label{eq_onesteptransferm}
B_{n}=\dfrac{1}{a_{n}}\begin{pmatrix} b_{n} - z & -\overline{a_{n-1}} \\
                 a_{n} & 0 \end{pmatrix}  ~\mbox{.}
\end{equation}
Here, the phrase ``almost singular'' refers to the matrix $a_n B_n$ whose determinant may come arbitrarily close to zero. For the ergodic counterpart in (\ref{eq_Jacobiop_2}) we even allow for the matrix $a_n B_n(\omega)$ to be singular, however since $a \in L^1$ this only happens on a set of zero probability. A more detailed discussion of singular Jacobi cocyles may be found in \cite{N}. 

As the main result of this article, we prove the following characterization of the singular(pp+sc)- spectrum for ergodic Jacobi matrices:
\begin{theorem} \label{thm_singspec}
Let $H_{\omega}$ be an ergodic family of Jacobi operators satisfying (\ref{eq_Jacobiop_2}) and (\ref{eq_stationarity}), and let $\ud \mu_{\omega,s}$ be the singular component of some spectral measure for $H_{\omega}$. Then, for $\mathbb{P}$-a.e. $\omega$ and $\omega^{\prime}$, $\ud \mu_{\omega,s}$ and $\ud \mu_{\omega^{'},s}$ are mutually singular on the set of positive Lyapunov exponent, $\{E: L(E) > 0\}$.
\end{theorem}
Theorem \ref{thm_singspec} generalizes a result in \cite{F} stated for the Schr\"odinger case. We note that this theorem will be proven using Kotani theory \cite{P,I}, which usually gives results relating to the absolutely continuous component of the spectrum. It is also noteworthy that mutual singularity of the singular component of spectral measures is reminiscent of an analogous property exhibited by rank 1 perturbations of bounded s.a. operators (see e.g. \cite{Q}).

We organize the paper as follows. Sec. \ref{sec_kotani} is devoted to the proof of Theorem \ref{thm_singspec}. As a crucial first step, we prove that energies with positive Lyapunov exponent (LE) allow for two solutions $u_{+}, u_{-}$ of $H_\omega \psi = E \psi$ which are $\mathbb{P}$-a.s. non-zero and $\mathit{l}^2$ (and hence necessarily exponentially decaying) at respectively $+\infty$ and $-\infty$ (Theorem \ref{thm_mfunk}). 

Consequently we obtain that given an energy $E$ with positive LE, the probability for $E$ to be in ``the natural'' supporting set (defined in (\ref{eq_singsuppset})) of the singular component of {\em{any}} spectral measure is zero (Theorem \ref{thm_singspec1}). This result is known for Schr\"odinger operators \cite{I}. 

While following the basic strategy laid out in \cite{I}, our proof even when applied to the Schr\"odinger case differs in a number of technical details. In particular, Theorem \ref{thm_mfunk} strengthens the respective statement for Schr\"odinger operators given in Theorem 6.5 of \cite{I}, by removing a set of zero probability, thus simplifying subsequent arguments (for details see remark  \ref{rem_strengthen}).  

Finally, Theorem \ref{thm_singspec} follows as a corollary of Theorem \ref{thm_singspec1}.

Sec. \ref{sec_dualityconj} is devoted to an application of Theorem \ref{thm_singspec} which really served as motivation for the results of Sec. \ref{sec_kotani}. Ergodic Jacobi operators with singular transfer matrices arise naturally in a tight binding description of solids subject to an external magnetic field. 

A prominent example is extended Harper's model \cite{D, E, G, H}, associated with the following almost-periodic operator $H_{\overline{\lambda},\theta;\alpha}$ defined on $l^2(\mathbb{Z})$:
\begin{eqnarray} \label{eq_hamiltonian}
& (H_{\overline{\lambda},\theta;\alpha} \psi)_k := v(\theta + \alpha k) \psi_{k} + c(\theta + \alpha k) \psi_{k+1} + \overline{c}(\theta + \alpha (k-1)) \psi_{k-1} ~\mbox{,}
\end{eqnarray}
where $\overline{\lambda}=(\lambda_{1},\lambda_{2},\lambda_{3})$ and
\begin{eqnarray} \label{eq_hamiltonian1}
& c_{\overline{\lambda}}(x) := \lambda_{3} \mathrm{e}^{-2\pi i (x+\frac{\alpha}{2})} + \lambda_{2} + \lambda_{1} \mathrm{e}^{2 \pi i (x+\frac{\alpha}{2})} ~\mbox{,}
~ v(x)  := 2 \cos(2 \pi x) ~\mbox{,}
\end{eqnarray}
are analytic functions on $\mathbb{R}/\mathbb{Z}$.

The operator in (\ref{eq_hamiltonian})  describes the influence of a transversal magnetic field to the motion of an electron in a 2-dimensional crystal layer. Here, $\alpha$ is a fixed irrational number associated with the magnetic flux through the unit cell and $\theta \in [0,1)$ is the (random) quasi-momentum usually referred to as ``random phase''. 

The parameter triple $(\lambda_{1},\lambda_{2},\lambda_{3})$ ({\em{coupling}}) models the lattice geometry as well as interactions between the lattice sites allowing for both nearest ($\lambda_2$) and next-nearest neighbor ($\lambda_1$ and $\lambda_3$) interactions. Without loss of generality one may assume  $0\leq \lambda_{2} ~\mbox{,} ~0 \leq \lambda_{1} + \lambda_{3}$ and at least one of $\lambda_{1} \mbox{,} ~\lambda_{2} \mbox{,} ~\lambda_{3}$ to be positive \cite{E}. We note that the quasi-periodic operator in (\ref{eq_hamiltonian}) specializes to the famous almost Mathieu operator (in physics literature also known as {\em{Harper's model}}) upon setting $\lambda_1=\lambda_3=0$. 

Even though some attempts, both rigorous \cite{D,E} and heuristic \cite{G, H}, have been made, so far little is known about extended Harper's model. As an important ingredient for the spectral analysis it is desirable to obtain the Lyapunov exponent (on the spectrum) as a function of the coupling constants. 

Theorem \ref{thm_singspec} enables us to prove a formula of the Lyapunov exponent in the so called ``dual regions'' of parameter space (Theorem \ref{thm_duality}). A precise meaning of the word ``duality'' is given at the beginning of Sec. \ref{sec_dualityconj}. We mention that Theorem \ref{thm_duality} provides a rigorous proof of results obtained in physics literature by Thouless \cite{G}. 

Finally, we mention that the missing link to a complete understanding of the Lyapunov exponent, the so called ``self dual'' region of coupling constants, a priori cannot be approached via duality. Computing the Lyapunov exponent in the self dual region therefore constitutes an important open problem which so far has escaped treatment even on a heuristic level in physics literature \cite{G}. The development of these a priori non-duality based methods will be the subject of future work.


{\bf Acknowledgement.} I would like to thank Svetlana Jitomirskaya for valuable discussions during the preparation of this manuscript. I would also like to thank the anonymous referee for insightful remarks leading to improvement of our presentation.

\section{Proof of the Main Theorem} \label{sec_kotani}

We start with the following basic result valid for {\em{any}} ergodic family of bounded self adjoint operators:
\begin{theorem} \label{thm_ergodicop}
For fixed $d \in \mathbb{N}$, let $H_{\omega}$ be any ergodic family of bounded self adjoint operators on $l^{2}(\mathbb{Z}^{d})$. Assume that for $\mathbb{P}$-a.e. $\omega$, the eigenvalues
of $H_{\omega}$ have finite multiplicity, whenever $H_{\omega}$ has non-empty point spectrum. Then, fixing $E \in \mathbb{R}$, $\mathbb{P}\{\omega: ~\mbox{E is eigenvalue for} ~H_{\omega}\} = 0$.
\end{theorem}
\begin{proof}
Let $B \in \mathbb{R}$ be any fixed Borel set and denote by $P_{B}(H_{\omega})$ the spectral projection onto $B$. 
The claim follows since by ergodicity $\dim \mathrm{Ran}{P_{B}(H_{\omega})}$ is constant $\mathbb{P}$-a.s with value either $\infty$ or $0$ \cite{C,M}.
\end{proof}

In order to relate the spectral properties of  the family $H_{\omega}$ to its dynamics given in terms of the Lyapunov exponent, we shall consider solutions $u(z)$ in $\mathbb{C}^{\mathbb{Z}}$ of the second order difference equation $H u = z u$, where $z \in \mathbb{C}$ is fixed. 
For two such solutions $u$ and $v$ (same $z$), define the Wronskian
\begin{equation}
W_{n}[u,v]:=u_{n} v_{n+1} - u_{n+1} v_{n} ~\mbox{,} n \in \mathbb{Z} ~\mbox{.} 
\end{equation}
It is easy to see that the Wronskian of two solutions satisfies the following conservation law
\begin{equation} \label{eq_conswron}
\overline{a_{n-1}} W_{n-1}[u,v] = a_{n} W_{n}[u,v] ~\mbox{,} n \in \mathbb{Z} ~\mbox{.} 
\end{equation}
As an immediate consequence of (\ref{eq_conswron}) we obtain
\begin{prop} \label{prop_pointsp}
Any eigenvalue of a Jacobi operator has multiplicity one.
\end{prop}
We note that Proposition \ref{prop_pointsp} uses $a_n \neq 0, \forall n \in \mathbb{Z}$. In particular, in the ergodic setup, remark \ref{rem_ergodic} implies that Proposition \ref{prop_pointsp} holds $\mathbb{P}$-a.s. Simplicity of the point spectrum of Jacobi operators shall be crucial for Sec. \ref{sec_dualityconj}. 

Sometimes it is useful to relate the solutions of the Schr\"odinger equation for the double-sided and the respective one-sided 
operator $H_{+}:=P_{+} H P_{+}$, where $P_{+}$ is a projection onto the positive half line. 

For $z \in \mathbb{C}$ fixed,
a solution $\psi$ to $H \psi = z \psi$ is a solution of the respective (positive) half-line equation if and only if it satisfies the boundary condition $\psi_0 = 0$. In particular this
implies that Proposition \ref{prop_pointsp} also applies for eigenvalues of $H^{+}$. 

More generally, for any $n \in \mathbb{Z}$ one can consider the positive $n$-one-sided operator $H_{+}^{n}:=P^{+n}HP^{+n}$ with $P^{+n}$ the projection onto $\{m > n\}$. For convenience, we allow $n=-\infty$ in which case we identify $H_{+}^{n}$ with the double-sided operator $H$. Similarly, one defines the negative $n$-one-sided operator, $H_{-}^{n}:= P^{-n} H P^{-n}$, $n \in \mathbb{Z}$, where $P^{-n}$ is the projection onto $\{m < n\}$. Again, we adopt the convention $H = H_{-}^{n}$ for $n=+\infty$. To simplify notation we write $H_{\pm} :=H_{\pm}^{0}$.

For the ergodic setup, Proposition \ref{prop_pointsp} and Theorem \ref{thm_ergodicop} imply the following statement which will be key for the proof of 
Theorem \ref{thm_singspec}:
\begin{prop} \label{prop_eigenvaluesofjacobi}
Let $H_{\omega}$ be a family of ergodic Jacobi operators, and let $E \in \mathbb{R}$ be fixed. Then the probability for $E$ to be an eigenvalue of some (positive or negative) 
$n$-one-sided operator (including $n=\pm \infty$) is zero.
\end{prop}

For $z \in \mathbb{H}^{+}$, let $\psi_{+}(z)$ be a solution to the equation $H \psi = z \psi$ which is $l^{2}$ at $+\infty$. Self adjointness of the half-line operator implies
existence of  $\psi_{+}(z)$ as well as $\psi_{+,n}(z) \neq 0, ~\forall n$. Using (\ref{eq_conswron}), we conclude uniqueness of $\psi_{+}(z)$ up to a multiplicative constant. Similarly, we can find a solution
$\psi_{-}(z)$ which is $l^{2}$ at $-\infty$. 

It is well known that the Green's function of $H$ can be expressed in terms of $\psi_{\pm}(z)$ by
\begin{equation} \label{eq_greens}
G(n,m;z) := \left \langle \delta_{n}, G(z) \delta_{m} \right \rangle = \dfrac{\psi_{-,\min\{n,m\}}(z) \psi_{+,\max\{n,m\}}(z)}{a_{m} W_{m}[\psi_{-}(z),\psi_{+}(z)]} ~\mbox{,}
\end{equation}
for $z \in \mathbb{C} \setminus \mathbb{R}$. Here, $\{\delta_{n}, n \in \mathbb{Z}\}$ denotes the standard basis in $l^{2}(\mathbb{Z})$.

Finally, for $n\in \mathbb{Z}$ we define the $m$-functions,
\begin{equation}
m_{+,n}(z):=-\dfrac{\psi_{+,n+1}(z)}{\overline{a_{n}} \psi_{+,n}(z)} ~\mbox{,} ~m_{-,n}(z):=-\dfrac{\psi_{-,n-1}(z)}{a_{n-1} \psi_{-,n}(z)} ~\mbox{,}
\end{equation}
where we set $m_{\pm}(z):=m_{\pm,0}(z)$. It is clear that in the ergodic case we have $m_{\pm,n}(\omega) = m_{\pm}(T^{n}\omega)$, $n\in \mathbb{Z}$. 
For simplicity of notation, we will suppress the $z$-dependence of the $m$-functions whenever the context permits.

Since $\psi_{\pm}$ are solutions to the Schr\"odinger equation, the $m$-functions satisfy the following Riccati-type equations
\begin{eqnarray} \label{eq_mfunct}
-\dfrac{1}{m_{+,n-1}} + (b_{n} - z) - \abs{a_{n}}^{2} m_{+,n} & = & 0  \nonumber \\
-\abs{a_{n-1}}^{2} m_{-,n} + (b_{n}-z) - \dfrac{1}{m_{-,n+1}} & = & 0 ~\mbox{.}
\end{eqnarray}

We mention that using a relation analogous to (\ref{eq_greens}) (see e.g. \cite{K, O}), the functions $m_{\pm}(z)$ relate to Borel transforms of the spectral measures $\ud \mu_{\pm}$ associated respectively with $H_{\pm}$ and $\delta_{\pm 1}$, i.e.
\begin{equation} \label{eq_mfuncBorel}
m_{\pm}(z) = \int \dfrac{\ud \mu_{\pm}(x)}{x-z} ~\mbox{.}
\end{equation}
This in particular implies that $m_{\pm}$ are Herglotz functions, and we may use the theory of Borel transforms of measures to consider limits as $\im(z) \to 0^{+}$.

From here on we specifically consider the ergodic set-up. The following statement relates the Lyapunov exponent and the $m$-functions. It extends a result known for
the Schr\"odinger case \cite{L}. 
\begin{prop} \label{prop_lyap}
For $\im{z} > 0$ we have
\begin{equation} \label{eq_lyap}
2 L(z) = \mathbb{E}_{\omega} \log \left\{1 + \dfrac{\im{z}}{\abs{a_{0}(\omega)}^{2} \im{m_{+}(\omega,z)}}\right\} ~\mbox{.}
\end{equation}
\end{prop}
\begin{remark}
We mention that a formula similar to (\ref{eq_lyap}) is given in \cite{O} (see Eq. (5.73) therein), however with $\abs{a_{0}(\omega)}^{2}$ in (\ref{prop_lyap}) being replaced by $a_{0}(\omega)$. Since a proof is omitted in \cite{O}, we correct the typo in \cite{O} providing the details below.
\end{remark}

\begin{proof}
Taking imaginary parts, the first equation in (\ref{eq_mfunct}) implies
\begin{eqnarray}
& \log\left\{\abs{a_{n}(\omega)}^{2} + \dfrac{\im{z}}{\im{m_{+}(T^{n}\omega)}} \right\} =  \log \im m_{+}(T^{n-1}\omega) & \nonumber \\
& - \log \im m_{+}(T^{n}\omega )  - 2 \log \abs{m_{+}(T^{n-1}\omega)} & \nonumber ~\mbox{.}
\end{eqnarray}
Since $T$ is measure preserving, taking expectations on both sides yields
\begin{eqnarray} \label{eq_1}
\mathbb{E}_{\omega} \log\left\{1 + \dfrac{\im{z}}{\abs{a_{0}(\omega)}^{2}\im{m_{+}(\omega)}} \right\} = - 2 \mathbb{E}_{\omega} 
\log \left \vert \dfrac{\psi_{+,1}(\omega)}{\psi_{+,0}(\omega)}  \right \vert  ~\mbox{.}
\end{eqnarray}

First, notice that by Birkhoff's ergodic theorem,
\begin{equation} \label{eq_2}
\dfrac{1}{n} \log \left \vert \dfrac{\psi_{+,n}(\omega)}{\psi_{+,0}(\omega)}  \right \vert \to 
\mathbb{E}_{\omega} \log \left \vert \dfrac{\psi_{+,1}(\omega)}{\psi_{+,0}(\omega)}  \right \vert ~\mbox{, $\mathbb{P}$-a.e.} ~\mbox{,}
\end{equation}
as $n \to \infty$.

On the other hand $\im{z} > 0$ implies $L(z) > 0$, hence using Oseledets' theorem (in its original GL(2,$\mathbb{C}$) formulation \cite{AA}) we obtain
\begin{equation} \label{eq_3}
\dfrac{1}{n} \log \left \vert \dfrac{\psi_{+,n}(\omega)}{\psi_{+,0}(\omega)}  \right \vert \to - L ~\mbox{, $\mathbb{P}$-a.e.} ~\mbox{,}
\end{equation}
as $n \to \infty$.
Thus, equations (\ref{eq_1}), (\ref{eq_2}), and (\ref{eq_3}) finally yield the claim.
\end{proof}

Using (\ref{eq_mfuncBorel}), the monotone convergence theorem implies the existence of
\begin{equation}
S_{\pm}(\omega,E) := \lim_{\epsilon \to 0^{+}} \dfrac{\im m_{+}(\omega, E + i \epsilon)}{\epsilon} ~\mbox{,}
\end{equation}
for any fixed $E \in \mathbb{R}$. 

Thus using dominated convergence, the formula for the Lyapunov exponent given in Proposition \ref{prop_lyap} persists in the limit, i.e.
\begin{equation} \label{eq_lyap}
2 L(E) = \mathbb{E}_{\omega} \log \left\{1 + \dfrac{1}{\abs{a_{0}(\omega)}^{2}} S_{+}(\omega,E)^{-1} \right\} ~\mbox{,}
\end{equation}
for any fixed $E \in \mathbb{R}$. Here, we also used that for any $E\in\mathbb{R}$, $L(E+i\epsilon)$ is continuous w.r.t. $\epsilon$ which e.g. can be obtained 
from the Thouless formula.

Since (\ref{eq_lyap}) contains the limit of the quantity $\frac{\im m_{+}(\omega,z)}{\im z}$, the following Lemma will prove to be of use:
\begin{lemma} \label{lem_1}
For $\im z > 0$,
\begin{equation}
\dfrac{\im m_{+}(\omega,z)}{\im z} = \dfrac{1}{\abs{a_{0}(\omega)}^{2} \abs{\psi_{+,0}(\omega,z)}^{2}} \norm{P_{+} \psi_{+}(\omega,z)}^{2} ~\mbox{.}
\end{equation}
\end{lemma}
\begin{proof}
Using functional calculus, Eq. (\ref{eq_mfuncBorel}) implies
\begin{equation}
\dfrac{\im m_{+}(\omega,z)}{\im{z}} = \sum_{n=1}^{\infty} \abs{\left \langle \delta_{1} , (P_{+} H_{\omega} P_{+} - z)^{-1} \delta_{n} \right \rangle}^{2}
\end{equation}
Finally, making use of expressions for the Green's function of the half line operator (see e.g. \cite{L}), we obtain
\begin{equation}
\abs{\left \langle \delta_{1} , (P_{+} H_{\omega} P_{+} - z)^{-1} \delta_{n} \right \rangle} = \left \vert \dfrac{\psi_{+,n}(\omega,z)}{a_{0}(\omega) \psi_{+,0}(\omega,z)} \right \vert ~\mbox{,}
\end{equation}
for any $n \geq 1$.
\end{proof}

We are now in the position to characterize the set of energies with positive Lyapunov exponent in terms of the $m$-functions. 
\begin{theorem} \label{thm_mfunk}
Let $E \in \mathbb{R}$ s.t. $L(E) > 0$. Then for $\mathbb{P}$-a.e. $\omega$, there exists a solution $u_{+}(\omega,E)$ with $u_{+,n}(\omega,E) \neq 0$ for all $n \in \mathbb{Z}$, which is $\mathit{l}^2$ at $+\infty$. In particular, $m_{+}(\omega,E+i0)$ exists, is finite and non-zero.
\end{theorem}
\begin{remark} \label{rem_strengthen}
In \cite{I}, the authors proof a similar statement for Schr\"odinger operators, which however does not rule out the case where the limiting solution $u_+(\omega,E)$ has zeros (Theorem 6.5 (b) \cite{I}). We mention that even though the authors of \cite{I} suspected this to only happen with zero probability, no proof was provided then.
\end{remark}

\begin{proof}
Using (\ref{eq_lyap}), $L(E) > 0$ implies that $S_{+}(\omega,E) < \infty$ on a set of positive measure $\mathbb{P}$; in particular, letting
$\Omega_{0}:=\{\omega: S_{+}(T^{n}\omega,E) < \infty ~\mbox{for some} ~n \in \mathbb{Z}\}$ we obtain $\mathbb{P}(\Omega_{0}) > 0$. 
In fact, since $T^{-1}\Omega_{0} = \Omega_{0}$, we have $\mathbb{P}(\Omega_{0}) = 1$ using ergodicity of $T$.

Let $\omega \in \Omega_{0}$ and $n \in \mathbb{Z}$ such that $S_{+}(T^{n} \omega,E) < \infty$. 
We mention that for $n=0$, we already obtain the statement using a standard argument from the the theory of Borel transforms.  

Using Lemma \ref{lem_1} for $z=E+i\epsilon$ and $\epsilon>0$,
\begin{equation}
\dfrac{\im m_{+}(T^{n}\omega,z)}{\im z} = \dfrac{1}{\abs{a_{0}(T^{n}\omega)}^{2}} \norm{P_{+} \psi_{+}(T^{n}\omega,z)}^{2} ~\mbox{,}
\end{equation}
where we define  $\psi_{+}$ to satisfy the condition $\psi_{+,0}(T^{n}\omega,z):=1$.

Then, by Lemma \ref{lem_1} and Banach Alaoglu ($S_{+}(T^{n} \omega,E) < \infty$), we obtain a limiting solution $u_{+}(\omega,E)$ of the Schr\"odinger equation,
\begin{equation}
w-\lim_{m\to\infty} \psi_{+}(T^{n} \omega,E+i \epsilon_{m}) =: u_{+}(\omega,E) ~\mbox{,}
\end{equation}
for some sequence $(\epsilon_{m})$, $\epsilon_{m} \searrow 0$. By construction  $u_{+}(\omega,E)$ is $l^{2}$ at $+\infty$.  Making use of (\ref{eq_conswron}),  any solution of the
Schr\"odinger equation which is $l^{2}$ at $+\infty$ is uniquely determined up to a multiplicative constant. Since by construction $u_{+}(n)=1$,
this limiting solution is in fact independent of the sequence $(\epsilon_{m})$, whence
\begin{equation}
 u_{+}(\omega,E) = w-\lim_{\epsilon \to 0^{+}} \psi_{+}(T^{n} \omega,E+i \epsilon) ~\mbox{.}
\end{equation}

If $u_{+,m}(\omega,E) = 0$ for some $m \in \mathbb{Z}$, $u_{+}(\omega,E)$ is an eigenfunction of $H_{+; \omega}^{m}$. 
By Proposition \ref{prop_eigenvaluesofjacobi} this however happens with zero probability. 

Thus, $\mathbb{P}$-a.s. we have $u_{+,m}(\omega,E) \neq 0$, $\forall m \in \mathbb{Z}$. In particular, taking the limit in the definition of $m_{+}$, we 
obtain the claim.
\end{proof}

Following Deift and Simon \cite{I}, for $\omega \in \Omega$ we consider the following set
\begin{equation} \label{eq_singsuppset}
\mathfrak{S}_{\omega}:=\left\{E \in \mathbb{R} : \limsup_{\epsilon \to 0^{+}} \left\{ \im G_{\omega}(0,0;E+i\epsilon) + \im G_{\omega}(1,1;E+i\epsilon)\right\}=\infty \right\} ~\mbox{.}
\end{equation}
It is well known that the singular component of any spectral measure for $H_{\omega}$ is supported on $\mathfrak{S}_{\omega}$ (Theorem of de la Vall\'ee-Poussin; see also \cite{O}). 

$G(0,0)$ and $G(1,1)$ relate to the m-functions by
\begin{eqnarray} \label{eq_greensfunk0011}
-G(0,0;z)^{-1} & = & \abs{a_{0}}^{2} m_{+}(z) + \abs{a_{-1}}^{2} m_{-}(z) + z - b_{0} \nonumber \\
G(1,1;z)^{-1} & = & m_{+}(z)^{-1} + \abs{a_{0}}^{2} \{ \abs{a_{-1}}^{2} m_{-}(z) + z -b_{0}\}^{-1} ~\mbox{.}
\end{eqnarray}
These relations are obtained directly from the expression for the Green's function given in (\ref{eq_greens}).

The following statement extends the statement of Proposition \ref{prop_eigenvaluesofjacobi} to $\mathfrak{S}_{\omega}$, the natural supporting set
of the singular component of any spectral measure for $H_{\omega}$. The respective result was first proven for Schr\"odinger operators by Deift and Simon
\cite{I}.
\begin{theorem} \label{thm_singspec1}
Let $E \in \mathbb{R}$ such that $L(E)>0$. Then, $\mathbb{P}\{\omega:E \in \mathfrak{S}_{\omega}\}=0$.
\end{theorem}
\begin{proof}
By Theorem \ref{thm_mfunk}, we know that for $\mathbb{P}$-a.e. $\omega$, the limits $m_{\pm}(\omega,E+i0)$ both exist, are finite and non-zero. 
Using the relations (\ref{eq_greensfunk0011}) for such $\omega$, $E \in \mathfrak{S}_{\omega}$ if and only if at least one of the following happens
\begin{eqnarray}
\abs{a_{0}(\omega)}^{2} m_{+}(\omega,E+i0) + \abs{a_{-1}(\omega)}^{2} m_{-}(\omega,E+i0) + E - b_{0}(\omega) & = &  0  \nonumber ~\mbox{,} \\
m_{+}(\omega,E+i0)^{-1} + \abs{a_{0}(\omega)}^{2} \{ \abs{a_{-1}(\omega)}^{2} m_{-}(\omega,E+i0) + E - b_{0}(\omega)\}^{-1} & = & 0  \nonumber ~\mbox{.}
\end{eqnarray}
Using the definition of $m_{\pm}$, it can easily be verified that both these situations imply that the respective limiting solutions $u_{\pm}(\omega,E)$ constructed in the proof of 
Theorem \ref{thm_mfunk} are in fact constant multiples of one another. Thus we obtain a solution which is $l^{2}$ at both $\pm \infty$, i.e. an eigenvector of $H_{\omega}$. 
Employing Proposition \ref{prop_eigenvaluesofjacobi}, the set of such $\omega$ has zero probability, which yields the claim.
\end{proof}

Finally we obtain Theorem \ref{thm_singspec}, which concludes this section.
\begin{proof}[Proof of Theorem \ref{thm_singspec}]
Using the definition of the Lyapunov exponent, the set $\mathfrak{L}:=\{E \in \mathbb{R}:0<L(E)<\infty\}$ is measurable. 
Consider, $\mathfrak{N}:=\left\{(\omega,E) \in \Omega \times \mathfrak{L} : E \in \mathfrak{S}_{\omega} \right\}$. 
Letting $f(\omega,z):=\im G_{\omega}(0,0;z) + \im G_{\omega}(1,1;z)$, where $(\omega,z) \in \Omega \times \left( \mathfrak{L} + i \mathbb{R}^{+} \right)$, we note that
$\mathfrak{N}$ can be written as $\mathfrak{N} = \{  (\omega,E) \in \Omega \times \mathfrak{L} : \limsup_{\epsilon \to 0^{+}} f(\omega,E+i\epsilon) = \infty \}$. 
Since $f$ is measurable, we conclude measurability of $\mathfrak{N}$.

The rest now follows using Tonelli. By Theorem \ref{thm_singspec1}, $\forall \omega^{\prime} \in \Omega:$
\begin{eqnarray}
0 & = & \int_{\mathfrak{L}} \mathbb{P}\{ \omega: E \in \mathfrak{S}_{\omega}\} \ud \mu_{\omega^{\prime}} \nonumber \\
  & = & \int \mu_{\omega^{\prime}}(\mathfrak{L} \cap \mathfrak{S}_{\omega}) \ud \mathbb{P}(\omega) ~\mbox{,} \nonumber
\end{eqnarray}
which implies the claim. 
\end{proof}

\section{An application: Duality and extended Harper's model} \label{sec_dualityconj}

In this section we illustrate the use of Theorem \ref{thm_singspec} to obtain an expression for the Lyapunov exponent in the dual regime of extended Harper's model . The model as well as its relevance has been discussed in Sec. \ref{sec_intro}.

Similar to the almost Mathieu operator, the Hamiltonian in (\ref{eq_hamiltonian}) exhibits an intrinsic symmetry in the coupling constants under Fourier transform, known as {\em{duality}}. Heuristically speaking, duality relates certain regions in parameter space, mapping localized states to Bloch waves. 

To give a precise meaning of duality, on the Hilbert space $\mathcal{H}^{\prime}:=L^{2}\left([0,1) \times \mathbb{Z}\right)$, consider the unitary operator $U$ defined by
\begin{equation} \label{eq_defunitary}
(U \phi)(\eta,m):=\sum_{n \in \mathbb{Z}} \int \ud \theta \mathrm{e}^{i 2 \pi m \theta} 
\mathrm{e}^{i 2 \pi n (m\alpha + \eta)} \phi(\theta,n) ~\mbox{.}
\end{equation}
We note that such operator has originally been introduced by Chulaevsky-Delyon \cite{BB} and later employed in \cite{F}, both in context of the almost Mathieu operator. 

For $\lambda_2 \neq 0$, let $\sigma$ denote the map, $\sigma(\overline{\lambda}):=\frac{1}{\lambda_{2}}(\lambda_{3},1,\lambda_{1})$, where
as earlier $\overline{\lambda}=(\lambda_{1},\lambda_{2},\lambda_{3})$ represents the coupling in extended Harper's model (see (\ref{eq_hamiltonian1})). 

We also define the following regions in parameter space:
\begin{description}
\item[region I] $0 \leq \lambda_{1}+\lambda_{3} \leq 1, ~0 \leq \lambda_{2} \leq 1$ ~\mbox{,}
\item[region II] $0 \leq \lambda_{1}+\lambda_{3} \leq \lambda_{2}, ~1 \leq \lambda_{2} $ ~\mbox{,}
\item[region III] $\max\{1,\lambda_{2}\} \leq \lambda_{1}+\lambda_{3}$ ~\mbox{.}
\end{description}

This partitioning is induced by the map $\sigma$ since
\begin{equation} \label{eq_dualityofregions}
\sigma(I^\circ) = II^\circ ~, ~\sigma(III) = III ~\mbox{.}
\end{equation}
In this sense, the interiors of region I and II are dual, whereas region III is self-dual. For a more general discussion of duality we refer to \cite{D}.

In \cite{E} (see Theorem 1 therein) it was proven that for Diophantine $\alpha$ and $\overline{\lambda} \in I^\circ$, the spectrum of $H_{\overline{\lambda},\theta;\alpha}$ 
is only pure point (i.e. for a certain set  $\mathbb{T}_{0}$ of phases $\theta$ with $\vert \mathbb{T}_{0} \vert =1$). Recall, $\alpha \in \mathbb{R}$ is called Diophantine if there exists $0 < b(\alpha)$ and $1 < r(\alpha)< + \infty$ s.t. for all $j \in \mathbb{Z}\setminus \{0\}$
\begin{equation} \label{eq_diophantine}
\abs{\sin(2\pi j \alpha)} > \dfrac{b(\alpha)}{\abs{j}^{r(\alpha)}} ~\mbox{.}
\end{equation}
It is known that $DC(r):=\{\alpha: \mbox{(\ref{eq_diophantine}) holds for} ~$r$ ~\mbox{and some} ~$b$\}$ is a set of full Lebesgue measure for every $r \geq 1$.

As we will argue, Theorem \ref{thm_singspec} will allow to relate the regions I and II establishing the following:
\begin{theorem} \label{thm_duality}
For every Diophantine $\alpha$, the LE in region II is zero on the spectrum. In region I, the Lyapunov exponent on the spectrum is positive and given by the formula,
\begin{equation} \label{eq_thoulessform}
L(\alpha; \overline{\lambda}, E) = \begin{cases}\log \left(  \dfrac{1+\sqrt{1 - 4\lambda_{1} \lambda_{3}}}{2 \lambda_{1}}\right) & \mbox{, if} ~\lambda_{1} \geq \lambda_{3}, ~\lambda_{2} \leq \lambda_{3} + \lambda_{1} ~\mbox{,}  \\
\log \left(  \dfrac{1+\sqrt{1 - 4\lambda_{1} \lambda_{3}}}{2 \lambda_{3}}\right) & \mbox{, if} ~\lambda_{3} \geq \lambda_{1}, ~\lambda_{2} \leq \lambda_{3} + \lambda_{1} ~\mbox{,}  \\
\log \left(  \dfrac{1+\sqrt{1 - 4\lambda_{1} \lambda_{3}}}{\lambda_{2} + \sqrt{\lambda_{2}^{2} - 4 \lambda_{1} \lambda_{3}}}\right) & ~\mbox{, if} ~\lambda_{2} \geq \lambda_{3} + \lambda_{1} ~\mbox{.} 
\end{cases}
\end{equation}
\end{theorem}

$H$ in (\ref{eq_hamiltonian}) naturally induces a bounded self adjoint operator on $\mathcal{H}^{\prime}$, which shall be denoted by $H_{\overline{\lambda};\alpha}^{\prime}$,\begin{equation}
H_{\overline{\lambda};\alpha}^{\prime} \phi := H_{\overline{\lambda};\theta, \alpha} \phi(\theta, .) ~, ~\forall \theta \in [0,1) ~\mbox{.}
\end{equation}

Using the definition of the unitary $U$, 
direct computation establishes the following relation between dual regions:
\begin{equation} \label{eq_duality}
U H_{\overline{\lambda};\alpha}^{\prime} U^{-1} = \lambda_{2} H_{\sigma(\overline{\lambda});\alpha}^{\prime} ~\mbox{.}
\end{equation}
To avoid confusion, we emphasize that (\ref{eq_duality}) shows that $H_{\overline{\lambda};\alpha,}^{\prime}$ and 
$H_{\sigma{(\overline{\lambda})};\alpha}^{\prime}$ are  unitarily equivalent and not the respective 
operators on the smaller space $\mathcal{H}=\mathit{l}^2(\mathbb{Z})$. 

An important property of duality is that preserves the density of states \cite{D}. To see this in the present formulation notice that
\begin{equation} \label{eq_invdos1}
U \delta_{0} = \delta_{0} ~\mbox{,}
\end{equation}
where $\delta_{0}$ is viewed as element of $\mathcal{H}^{\prime}$, i.e. employing the embedding $\mathcal{H} \hookrightarrow \mathcal{H}^\prime$
\begin{equation}
(u_n) \mapsto (\phi(\theta, n)) ~\mbox{, where} ~\phi(\theta,n) = u_n ~\forall n \in \mathbb{Z} ~\mbox{,} ~\theta \in [0,1) ~\mbox{.}
\end{equation}

Denote by $\ud n(\left\{H_{\overline{\lambda};\theta, \alpha}\right\}; E)$ the density of states of the ergodic family $\left\{H_{\overline{\lambda};\theta, \alpha}\right\}$ defined as usual by
\begin{equation}
\mathbb{E}_{\theta} \left\langle  \delta_{0}, f(H_{\overline{\lambda},\theta;\alpha}) \delta_{0} \right\rangle =: \int f(E) \ud n(\left\{H_{\overline{\lambda};\theta, \alpha}\right\}; E)    ~\mbox{,}
\end{equation}
for $f$ continuous and compactly supported on $\mathbb{R}$. Then, combining (\ref{eq_duality}) and (\ref{eq_invdos1}), we obtain the desired invariance of the density of states,
\begin{equation} \label{eq_inv}
\ud n\left(\left\{H_{\overline{\lambda},\theta;\alpha}\right\}; E\right) = \ud n\left( \left\{ \lambda_{2}  H_{\sigma(\overline{\lambda}),\theta;\alpha}\right\}; E\right) = \ud n\left( \left\{H_{\sigma(\overline{\lambda}),\theta;\alpha}\right\}; \lambda_2^{-1} E\right)
 ~\mbox{.}
\end{equation}
We mention that the last equality in (\ref{eq_inv}) follows from the spectral mapping theorem.

In \cite{F} it is shown that fixing $\overline{\lambda}$ and $\alpha$, it is possible to find a labeling of the eigenfunctions 
of $H_{\overline{\lambda},\theta;\alpha}$ that produces measurable functions w.r.t. the random phase $\theta$. 

The proof in \cite{F} (Theorem 2.2) is given for an arbitrary bounded, ergodic Schr\"odinger operator on $l^{2}(\mathbb{Z})$, the crucial property needed being however only simplicity of the point spectrum. 
As shown in Proposition \ref{prop_pointsp}, the same is true for Jacobi operators. Therefore, Theorem 2.2. in \cite{F} extends to an ergodic family $H_{\omega}$ of Jacobi operators, as defined in Sec. \ref{sec_kotani}.

A measurable labeling of eigenfunctions is obtained as follows:  For $\omega \in \Omega$, let $u$ be a normalized  eigenvector of $H_{\omega}$.  Pick the unique $j \in \mathbb{Z}$ such
that $\abs{u(j)} \geq \abs{u(k)}, ~\forall k$ and $\abs{u(j)} > \abs{u(k)}, ~\forall k < j$ (``left-most maximum''). We say $u$ is attached to $j$. Let $N_{j}(\omega)$ 
be the number of eigenfunctions attached to $j$, and for $j,k \in \mathbb{Z}$ set $\Omega_{j,k}:=\{\omega : N_{j}(\omega) \geq k \}$. On $\Omega_{j,1}$, let 
$(u_{1}(n;\omega,j))$ be the eigenfunction attached to $j$ with maximal $j$th component (in magnitude) and $e_{1}(\omega,j)$ its associated
eigenvalue. If multiple eigenfunctions attached to $j$ have the same magnitude of their $j$th entry, discriminate them by choosing the one with largest eigenvalue.
Note that since the point spectrum is simple, there can only be finitely many different eigenfunctions attached to $j$ whose $j$th components coincide in magnitude.  
For further details we refer to \cite{F}.

For this labeling of eigenfunction we have (see Sec. 2 in
\cite{F}),
\begin{theorem} \label{thm_lana1}
Let $H_{\omega}$ be an ergodic family of Jacobi operators defined as in Sec. \ref{sec_kotani}, and $N_{j}(\omega), e_{l}(\omega,j)$, and $(u_{l}(n;\omega,j))_{n \in \mathbb{Z}}$
defined as above. Then, fixing $l, j$ and $n$,  $N_{j}(\omega), e_{l}(\omega,j)$, and $u_{l}(n;\omega,j)$ are measurable functions on $\Omega$.
\end{theorem}

The next statement shows that applying the unitary $U$ defined in (\ref{eq_defunitary}), we obtain spectral measures that are constant for a.e. phase. 
The proof carries over from \cite{F} (see Theorem 3.4 therein) without any modification.
\begin{theorem} \label{thm_lana2}
Fix $\overline{\lambda}$ and $\alpha$ Diophantine and let $(u_{l}(n;\theta,j))$ be the measurable labeling of eigenfunctions for $H_{\sigma{(\overline{\lambda})},\theta;\alpha}$
described in Theorem \ref{thm_lana1}. For $f(\theta) \in L^{2}([0,1],\ud \theta)$ arbitrary and some fixed $l,j \in \mathbb{Z}$, set $\phi(\theta,n):=f(\theta) u_{l}(n;\theta,j)$.
Denote by $\ud \mu_{\eta}$ the spectral measure for $H_{\overline{\lambda},\theta;\alpha}$ and $\psi_{\eta}(n):=(U\phi)(\eta,n)$. Then. $\ud \mu_{\eta}$ is
$\eta$ independent for a.e. $\eta$.
\end{theorem}

Equipped with Theorem \ref{thm_lana2}, we are now ready to prove Theorem \ref{thm_duality}.
\begin{proof}[Proof of Theorem \ref{thm_duality}]
For a fixed Diophantine frequency $\alpha$, let $\mathfrak{L}_{\overline{\lambda}}^{+}:=\{E \in \mathrm{spec}: L(\alpha; \overline{\lambda}, E) > 0\}$.
We claim that $\mathfrak{L}_{\overline{\lambda}}^{+} = \emptyset $ for any $\overline{\lambda}$ in region II. 

First consider $\overline{\lambda} \in II^\circ$ and let $\ud \mu_{\eta}$ be defined according to Theorem \ref{thm_lana2}. Then, since $\ud \mu_{\eta}$ is constant a.e. $\eta$, Theorem \ref{thm_singspec} implies that in fact $\ud \mu_{\eta}$ is purely a.c. on $\mathfrak{L}_{\overline{\lambda}}^{+}$ for a.e. $\eta$. Hence, using a well known argument by Pastur-Ishii (\cite{B,C}; see also the proof of Theorem 9.13 in \cite{M}) we obtain $P_\eta^{\mathrm{ac}}(\mathfrak{L}_{\overline{\lambda}}^{+}) = P_\eta(\mathfrak{L}_{\overline{\lambda}}^{+}) = 0$; here $P_\eta(B)$ and $P_\eta^{\mathrm{ac}}(B)$ denotes the spectral projection associated with a Borel set $B$ and its restriction to the ac subspace of $H_\eta$, respectively. In particular, in terms of the density of states this implies
 \begin{equation*}
 \mathrm{spec} = \mathrm{supp} \left\{\ud n\left(\left\{H_{\overline{\lambda},\theta;\alpha}\right\}\right) \right\} \subseteq \overline{\mathbb{R} \setminus \mathfrak{L}_{\overline{\lambda}}^{+}} ~\mbox{,}
\end{equation*}
which by continuity of the Lyapunov exponent w.r.t. to $E$ \cite{N, A} yields the claim for $\overline{\lambda} \in II^\circ$. To obtain the claim for all off region II, 
i.e. including the boundary, we employ the following continuity argument: 

Let $\overline{\lambda}$ be an arbitrary point in $\partial II$ and $E$ a fixed energy in the spectrum of $H_{\overline{\lambda}}$.  Using continuity of the spectrum in the Hausdorff metric \footnote{It is a well known fact that given a sequence $(A_n)$ of bounded self adjoint operators approximating a bounded operator $A$ in {\em{norm topology}}, the spectra satisfy $\mathrm{spec}(A_n) \to \mathrm{spec}(A)$ in the Hausdorff metric.}, there exists a sequence $(\overline{\lambda}_n)$, approximating $\overline{\lambda}$ from within $II^\circ$, and a corresponding sequence of energies $(E_n)$, with $E_n$ in the spectrum of $H_{\overline{\lambda}_n}$, such that $E_n \to E$. Since the Lyapunov exponent is jointly continuous at $(E,\overline{\lambda})$ for a fixed Diophatine $\alpha$ \cite{N}, we obtain $L(\alpha; \overline{\lambda}, E) = 0$, as claimed.

Making use of Thouless' formula and (\ref{eq_inv}), zero Lyapunov exponent in region II immediately implies an expression for the dual region since
\begin{eqnarray} \label{eq_finalproof}
L(\alpha; \overline{\lambda}, E) & = & - \int \log \abs{c_{\overline{\lambda}}(x)} \ud x + \int \log\abs{E-E^{\prime}} \ud n\left(\left\{H_{\overline{\lambda},\theta;\alpha}\right\}; E^\prime \right) \nonumber \\
 & = & \int \log \left \vert \frac{\lambda_{2} c_{\sigma(\overline{\lambda})}(x)}{c_{\overline{\lambda}}(x)} \right \vert \ud x+ L(\alpha; \sigma(\overline{\lambda}), \lambda_2^{-1} E) ~\mbox{.}
\end{eqnarray}
Here, $c_{\overline{\lambda}}$ is defined as in (\ref{eq_hamiltonian1}). 

Using Jensen's formula, the integral in (\ref{eq_finalproof}) can be computed explicitly \cite{E} (see Sec. 2, p. 107), 
\begin{equation} \label{eq_integral}
\int \log \abs{c_{\overline{\lambda}}(x)} \ud x = \begin{cases} \log \lambda_{3} & \mbox{if} ~\lambda_{3}  \geq  \lambda_{1}  \geq 0 \\ & ~\mbox{and} ~\lambda_{1}  + \lambda_{3} \geq \lambda_{2} \geq 0 ~\mbox{,} \\
\log{\lambda_{1}} & \mbox{if} ~\lambda_{1}  \geq \lambda_{3}  \geq 0  \\& ~\mbox{and} ~\lambda_{1} + \lambda_{3}  \geq \lambda_{2} \geq 0 ~\mbox{,} \\
\log \left \vert \dfrac{2\lambda_{1}\lambda_{3}}{-\lambda_{2}+\sqrt{\lambda_{2}^2 - 4\lambda_{1}\lambda_{3}}} \right \vert & \mbox{if} ~\lambda_{1} +\lambda_{3}  \leq \lambda_{2} ~\mbox{and} ~ \lambda_{1},\lambda_{3} \neq 0 ~\mbox{,}\\
\log \lambda_{2} & \mbox{if} ~\lambda_{1} + \lambda_{3}  \leq \lambda_{2} ~\mbox{,} ~ \lambda_{1} ~\mbox{or}  ~\lambda_{3} = 0 ~\mbox{,} \end{cases}
\end{equation}
giving rise to the expression for the Lyapunov exponent in region I given in (\ref{eq_thoulessform}).
\end{proof}

\bibliographystyle{amsplain}

\end{document}